\documentclass{amsart}
\usepackage{amssymb}
\usepackage{amsfonts}
\usepackage{amssymb}
\usepackage{amsmath}
\usepackage{amsthm}
\usepackage{enumerate}
\usepackage{tabularx}
\usepackage{centernot}
\usepackage{mathtools}
\usepackage{stmaryrd}
\usepackage{amsthm,amssymb}
\usepackage{etoolbox}
\usepackage{amssymb}
\usepackage{amsfonts}
\usepackage{amssymb}
\usepackage{amsmath}
\usepackage{amsthm}
\usepackage{enumerate}
\usepackage{pdfpages}
\usepackage{tikz}
\usepackage{amsmath}
\usepackage{amssymb}
\usetikzlibrary{arrows.meta}
\usepackage{tabularx}
\usepackage{centernot}
\usepackage{dutchcal}
\usepackage{mathtools}
\usepackage{stmaryrd}
\usepackage{amsthm,amssymb}
\usepackage{etoolbox,color}
\usepackage{tikz}
\usepackage{amssymb}
\usetikzlibrary{matrix}
\usepackage{tikz-cd}
\usepackage{tikz}
\usepackage{afterpage}
\usepackage{url}
\usepackage{tikz}
\usepackage{amssymb}
\usetikzlibrary{matrix}
\usepackage{tikz-cd}
\usepackage{tikz}
\usepackage{marginnote}
\definecolor{mygray}{gray}{0.85}
\usepackage[backgroundcolor=mygray,colorinlistoftodos,prependcaption,textsize=small]{todonotes}
\usepackage{xargs}       
\usepackage{float}
\usepackage{graphics}
\usepackage{relsize}

 \tikzcdset{scale cd/.style={every label/.append style={scale=#1},
 		cells={nodes={scale=#1}}}}
 
 \usepackage[colorlinks,citecolor=blue,urlcolor=blue, linkcolor=blue]{hyperref}

\usepackage[nocompress]{cite}
\newcommand{\mrm}[1]{\mathrm{#1}}

\renewcommand{\leq}{\leqslant}
\renewcommand{\geq}{\geqslant}

\newcommand{\fleq}{\leqslant_{\ast}^{\mrm{ff}}}

\newcommand{\bigast}{\mathop{\LARGE  \mathlarger{\mathlarger{*}}}}

\newcommand{\type}{\mathrm{tp}}

\makeatletter
\def\subsection{\@startsection{subsection}{3}%
  \z@{.5\linespacing\@plus.7\linespacing}{.3\linespacing}%
  {\bfseries\centering}}
\makeatother

\makeatletter
\def\subsubsection{\@startsection{subsubsection}{3}%
  \z@{.5\linespacing\@plus.7\linespacing}{.3\linespacing}%
  {\centering}}
\makeatother

\makeatletter
\def\myfnt{\ifx\protect\@typeset@protect\expandafter\footnote\else\expandafter\@gobble\fi}
\makeatother

\renewcommand{\restriction}{ {\upharpoonright} }

\def\Ind{\setbox0=\hbox{$x$}\kern\wd0\hbox to 0pt{\hss$\mid$\hss}
	\lower.9\ht0\hbox to 0pt{\hss$\smile$\hss}\kern\wd0}
\def\Notind{\setbox0=\hbox{$x$}\kern\wd0\hbox to 0pt{\mathchardef
		\nn=12854\hss$\nn$\kern1.4\wd0\hss}\hbox to
	0pt{\hss$\mid$\hss}\lower.9\ht0 \hbox to 0pt{\hss$\smile$\hss}\kern\wd0}

\def\ind{\mathop{\mathpalette\Ind{}}}

\newtheorem{theorem}{Theorem}[section]

\theoremstyle{definition}

\newtheorem{corollary}[theorem]{Corollary}
\newtheorem{definition}[theorem]{Definition}
\newtheorem{lemma}[theorem]{Lemma}
\newtheorem{proposition}[theorem]{Proposition}

\newtheorem{question}[theorem]{Question}
\newtheorem{observation}[theorem]{Observation}
\newtheorem{fact}[theorem]{Fact}
\newtheorem{context}[theorem]{Context}

\newtheorem{remark}[theorem]{Remark}
\newtheorem*{theorem*}{Main Theorem 1}
\newtheorem*{fact*}{Fact}
\newtheorem*{theorem2*}{Main Theorem 2}
\newtheorem*{theorem3*}{Main Theorem 3}

\newtheorem{claim}[theorem]{Claim}
\newtheorem{notation}[theorem]{Notation}

\newtheorem{construction}[theorem]{Construction}

\usepackage[backgroundcolor=mygray,colorinlistoftodos,prependcaption,textsize=small]{todonotes}
\usepackage{xcolor}

\newcommand{\pureindep}[1][]{%
	\mathrel{
		\mathop{
			\vcenter{
				\hbox{\oalign{\noalign{\kern-.3ex}\hfil$\vert$\hfil\cr
						\noalign{\kern-.7ex}
						$\smile$\cr\noalign{\kern-.3ex}}}
			}
		}\displaylimits_{#1}
	}
}

\newcommand{\indep}[2]{%
	\mathrel{
		\mathop{
			\vcenter{
				\hbox{%
					\oalign{
						\noalign{\kern-.3ex}\hfil$\vert$\hfil\cr
						\noalign{\kern-.7ex}
						$\smile$\cr\noalign{\kern-.3ex}
					}
				}
			}
		}^{\!\!\!\!\!#2}_{\!\!\hspace{-0.1em}#1}
	}
}

\newcommand{\displayindep}[2]{%
	\mathrel{
		\mathop{
			\vcenter{
				\hbox{%
					\oalign{
						\noalign{\kern-.3ex}\hfil$\vert$\hfil\cr
						\noalign{\kern-.7ex}
						$\smile$\cr\noalign{\kern-.3ex}
					}
				}
			}
		}^{\!\!\hspace{-0.1em}#2}_{\!\!\hspace{-0.1em}#1}
	}
}

\newcommand{\displayfindep}[2]{%
	\mathrel{
		\mathop{
			\vcenter{
				\hbox{%
					\oalign{
						\noalign{\kern-.3ex}\hfil$\vert$\hfil\cr
						\noalign{\kern-.7ex}
						$\smile$\cr\noalign{\kern-.3ex} 
					}
				}
			}
		}^{\!\hspace{-0.14em}#2}_{\!\!\hspace{-0.05em}#1}
	}
}

\newcommand{\astindep}[1][]{\indep{#1}{*}}

\newcommand{\free}{\mathcal{F}_\mathbf{V}}
\newcommand{\freeinf}{\mathcal{F}^{\infty}_\mathbf{V}}

\newcommand{\freefg}{\mathcal{F}^{<\infty}_\mathbf{V}}

\newtheoremstyle{TheoremNum}
{\topsep}{\topsep}              
{\itshape}                      
{}                              
{\bfseries}                     
{.}                             
{ }                             
{\thmname{#1}\thmnote{ \bfseries #3}}
\theoremstyle{TheoremNum}
\newtheorem{thmn}{Theorem}

\newtheoremstyle{DefNum}
{\topsep}{\topsep}              
{}                      
{}                              
{\bfseries}                     
{.}                             
{ }                             
{\thmname{#1}\thmnote{ \bfseries #3}}
\theoremstyle{DefNum}
\newtheorem{defn}{Definition}

\newtheoremstyle{CorNum}
{\topsep}{\topsep}              
{}                      
{}                              
{\bfseries}                     
{.}                             
{ }                             
{\thmname{#1}\thmnote{ \bfseries #3}}
\theoremstyle{CorNum}
\newtheorem{corl}{Corollary}

\begin{document}

\begin{abstract} 
We investigate the relationship between the Eklof-Mekler-Shelah Construction Principle for a variety of algebras $\mathbf{V}$ and the question of superstability of the free objects in $\mathbf{V}$, denoted as $\free$. We consider this question in the general setting of $\mrm{AEC}$-coverings of $\free$, with applications to first-order logic and beyond. Our main result is that if a strong form of the Construction Principle is satisfied, then almost all $\mrm{AEC}$-covering of $\free$ are unsuperstable. Concrete applications to $R$-modules and varieties of groups are also considered.
\end{abstract}

\title[Construction principle and superstability]{The construction principle and superstability of free objects in varieties of algebras}

\thanks{The authors were supported by project PRIN 2022 ``Models, sets and classifications", prot. 2022TECZJA. The second listed author was also supported by INdAM Project 2024 (Consolidator grant) ``Groups, Crystals and Classifications''. The third listed author was also supported by an INdAM post-doc grant.}

\author{Tapani Hyttinen}
\author{Gianluca Paolini}
\author{Davide E. Quadrellaro}

\address{Department of Mathematics and Statistics, University of Helsinki, P.O. Box 68 (Pietari Kalmin katu 5), 00014 Helsinki, Finland}
\email{tapani.hyttinen@helsinki.fi}

\address{Department of Mathematics “Giuseppe Peano”,
	University of Torino,
	Via Carlo Alberto 10,
	10123 Torino, Italy.}
\email{gianluca.paolini@unito.it}
\email{davideemilio.quadrellaro@unito.it}

\address{Istituto Nazionale di Alta Matematica ``Francesco Severi'',
	Piazzale Aldo Moro 5
	00185 Roma, Italy.}
\email{quadrellaro@altamatematica.it}


\date{\today}
\maketitle

\section{Introduction}\label{sec:introduction}

One of the main dividing lines in Shelah's Classification Theory \cite{Sh} is the notion of superstability, this is a tameness notion related to the number of first-order types, which admits many equivalent definitions. It is easy to see that free abelian groups of infinite rank are not superstable. Similarly, it was proved by Gibone in his thesis (cf.~\cite[Cor.~1.8]{Baudisch}) and by Poizat \cite{Poizat} that non-abelian free groups are not superstable\footnote{Much more general proofs of this fact are now known, see e.g. \cite{dahmani,houcine,MPS22}.}. All this motivated us to look at the following question:
	
	\begin{question}\label{first_question} Let $\mathbf{V}$ be a variety of algebras (in the sense of universal algebra). When is the first-order theory of the $\mathbf{V}$-free objects of infinite rank superstable?
\end{question}
	
	The restriction to objects of infinite rank in \ref{first_question} is simply to ensure that all these structures are elementary equivalent (in first-order logic), as e.g. free abelian groups of different finite ranks are not elementary equivalent. We introduce a little bit of notation. Let $\mathbf{V}$ be a variety of algebras (in the sense of universal algebra). We denote by $\free$ (resp. $\freeinf$) the free objects in $\mathbf{V}$ (resp. the free objects of infinite rank in $\mathbf{V}$).  For $A, B \in \free$ we write $A \fleq B$ if there is $C \in \free$ such that $A \ast C = B$.  Although our superstability question was phrased in terms of first-order logic, it admits a natural generalisation to abstract elementary classes (AECs) \cite{Sh2}. AECs are one of the most general frameworks designed to develop model theory beyond first-order logic. This leads us to the following notion of independent interest.

	\begin{definition}\label{def:AEC_covering} We say that $(\mathcal{K}, \preccurlyeq)$ is an \emph{$\mrm{AEC}$-covering} of $(\free, \fleq)$ when the following conditions are satisfied:
	\begin{enumerate}[(1)]
	\item $\freeinf\subseteq \mathcal{K}$;
	\item $(\mathcal{K}, \preccurlyeq)$ is an $\mrm{AEC}$;
	\item for $A, B \in \free\cap \mathcal{K}$, $A \fleq B$ implies $A \preccurlyeq B$;
	\item $(\mathcal{K}, \preccurlyeq)$ has $\mrm{AP}$ and $\mrm{JEP}$.
		\end{enumerate}
\end{definition}

	Notice that, letting $T$ be the first-order theory of $\freeinf$ (which is a complete theory) we have that $(\mrm{Mod}(T), \preccurlyeq_{\omega, \omega})$ is an AEC covering of $(\free, \fleq)$, where here $\preccurlyeq_{\omega, \omega}$ denotes the relation of elementary substructure (in first-order logic). Interestingly, in this setting we can ask the following more general variant of Question~\ref{first_question}.
	
	\begin{question}\label{second_question} Let $\mathbf{V}$ be a variety of algebras and let $(\mathcal{K}, \preccurlyeq)$ be an $\mrm{AEC}$-covering of $(\free, \fleq)$. When is it the case that $(\mathcal{K}, \preccurlyeq)$ is superstable?
\end{question}
A few words of explanations are due here concerning Question~\ref{second_question}. In the context of AECs there are various notions of superstability. For AECs with amalgamation, joint embedding and arbitrarily large models (a commonly studied framework) it was shown in \cite{grossberg} that all these alternative definitions of superstability are equivalent under the assumption of \emph{tameness}, which in the light of \cite{boney} can be seen essentially as a set-theoretic assumption. We stress that, in our main result \ref{main_th1}, we refer to the classical notion of superstability in terms of the number of types, and so such set-theoretic assumptions are not necessary for our purposes. 

Answering questions on $\mathbf{V}$-free objects for all varieties of algebras $\mathbf{V}$ is notoriously difficult. One landmark result in the logic literature, due to Eklof, Mekler and Shelah, is the following tripartite result, which is \cite[Theorem~16]{MSh} and builds on \cite{EM} (see also the related references \cite{EM2,mekler,mekler2,MSh,MSh3} on this topic).
	
	\begin{fact}[Eklof-Mekler-Shelah]\label{recap_fact}
	For any variety of algebras $\mathbf{V}$ in a countable language exactly one of the following possibilities holds:
	\begin{enumerate}[(i)]
		\item there is an $\mathfrak{L}_{\infty,\omega_1}$-free algebra of cardinality $\aleph_1$ which is not free;
		\item every $\mathfrak{L}_{\infty,\omega_1}$-free algebra of cardinality $\aleph_1$ is free, and for every infinite cardinal $\kappa$ there is an $\mathfrak{L}_{\infty,\kappa}$-free algebra of cardinality $\kappa^+$ which is not free;
		\item every $\mathfrak{L}_{\infty,\omega}$-free algebra is free.
	\end{enumerate}
\end{fact}

The theorem above concerns the question of axiomatisability of $\mathbf{V}$-free objects in infinitary languages, but behind \ref{recap_fact}(i) there is actually a combinatorial pattern satisfied by the $\mathbf{V}$-free objects of countable rank that goes under the name of \emph{Construction Principle} (CP) (first isolated in \cite{EM}).\footnote{A new look at the Construction Principle was recently given by the authors of the present paper in \cite{HPQ}, but as this is the not strictly related to the superstability question studied here we will not dwell further on the framework and the results from \cite{HPQ}.} Our way of attacking Questions~\ref{first_question} and \ref{second_question} is to approach the problem through the following more specific question.
	
\begin{question} How does the Construction Principle relate to the problem of superstability of $\mathbf{V}$-free objects (in the sense of Questions~\ref{first_question} and \ref{second_question})?
\end{question} 

	We will see that this way of looking at the problem yields strong non-superstability results, which connect to classical results on first-order logic, as well as recent results on AECs of $R$-modules from works of Mazari-Armida \cite{armida1,armida3,armida4}.

\subsection{The first venue}\label{intro:first_venue}
	
The main result of the paper asks for the satisfaction of a slight strengthening of the Construction Principle, which we will term the \emph{Reinforced Construction Principle} (RCP). To avoid technical complications, we introduce the RCP only for varieties in a countable language, although it is possible to extend its statement and also our results to the general setting of languages of arbitrary size.	
	
\begin{definition}\label{RCP_intro}\label{RCP_proof}
	Let $\mathbf{V}$ be a countable variety of algebras. We say that $\mathbf{V}$ satisfies the \emph{Reinforced Construction Principle $(\mrm{RCP})$} if the following conditions are met:
	\begin{enumerate}
		\item[$ (\mrm{RCP}) $] there are countable $A, B \in \freeinf$ with $A \leq B$ such that:
		\begin{enumerate}[(a)]
			\item $A$ is freely generated by $\{a_i : i < \omega\}$, and for every $i < \omega$, letting $A_i = \langle a_j : j \leq i \rangle_B$, we have that $A_i \fleq B_i$;
			\item $A \not \fleq B \ast C$, for every $C \in \free$;
			\item $B$ is freely generated by $\{b_i : i < \omega\}$ and $\mrm{dcl}^{\mrm{qf}}_B(Ab_0) = B$ (where the operator $\mrm{dcl}^{\mrm{qf}}$ denotes the quantifier-free definable closure).
		\end{enumerate}
	\end{enumerate}
\end{definition}

The classical Construction Principle CP is obtained from RCP by simply dropping the third condition  ``$\mrm{dcl}^{\mrm{qf}}_B(Ab_0) = B$''. Crucially, in concrete cases it often happens that when CP holds, then RCP holds as well; we will see that this is the case for many varieties of groups and modules. Our main result is the following.

\begin{theorem}\label{main_th1} Let $\mathbf{V}$ be a countable variety of algebras which satisfies $\mrm{RCP}$ (cf. \ref{RCP_intro}) and let $(\mathcal{K}, \preccurlyeq)$ be a strong enough (cf. \ref{def_enough_strong}) $\mrm{AEC}$-covering of $(\free, \fleq)$. Then $(\mathcal{K}, \preccurlyeq)$ is not superstable (i.e., $(\mathcal{K}, \preccurlyeq)$ is not $\lambda$-stable on a tail of cardinals).
\end{theorem}
	
In Theorem~\ref{main_th1} we have one further assumption which needs justification, i.e., that of ``$(\mathcal{K}, \preccurlyeq)$ being a strong enough AEC-covering of $(\free, \fleq)$''. This condition is a bit technical and we refer to \ref{def_enough_strong} for details, but we observe here the following. Firstly, this condition is trivially satisfied when the covering is $(\mrm{Mod}(T), \preccurlyeq_{\omega, \omega})$, for $T$ the first order theory of algebras in $\freeinf$. Secondly, and most importantly, we stress that {\em some assumptions on the submodel relation $\preccurlyeq$ are necessary}, since, for instance, the AEC of abelian groups with the relation of substructure is $\omega$-stable, despite the fact that the variety of abelian groups satisfies RCP (for more on this see Section~\ref{subsec1_intro}). Crucially, Theorem~\ref{main_th1} is easy to verify in practice, thereby allowing us to obtain some results concerning free objects in varieties of $R$-modules and groups (see Sections~\ref{subsec1_intro}-\ref{subsec2_intro}). We expect that our results will apply to other algebraic contexts. We stress that RCP is a completely algebraic statement on the variety and no knowledge of model theory is required to understand its content.

\subsubsection{$R$-modules}\label{subsec1_intro}

Obviously, given a ring $R$, the class of $R$-modules form a variety and so our setting naturally connects with the model theory of $R$-modules, and in particular with the study of AECs of $R$-modules, recently undertaken by Mazari-Armida and others with great success, see e.g. \cite{armida1,armida2,armida3,armida4,Armida_ros,Trlifaj,SaTr,Baldwin, boneyLmodule}. The most relevant result to our context is the following theorem by Mazari-Armida, from \cite{armida3}.
	
\begin{fact}[Mazari-Armida]\label{Armida_fact_nonsuper} Let $R$ be a ring, let $\mathcal{Flat}_R$ be the class of flat $R$-modules and let $\leq_{\mrm{pp}}$ be the relation of pure submodule. Then $(\mathcal{Flat}_R, \leq_{\mrm{pp}})$ is a superstable AEC if and only if $R$ is left-perfect.
\end{fact}

Notice that, letting $\mathbf{V}_R$ be the variety of $R$-modules, obviously $(\mathcal{Flat}_R, \leq_{\mrm{pp}})$ is an AEC-covering of $(\mathcal{F}_{\mathbf{V}_R}, \fleq)$, and, as already observed by Mazari-Armida in \cite[Ex.~3.24]{armida1}, if the ring is not right-coherent, then this AEC-covering is not first-order. Furthermore, as well-known (see e.g. \cite{EM2}), $\mathbf{V}_R$ satisfies CP if and only if $R$ is not left-perfect. Thus Mazari-Armida's result fits perfectly within our setting and it is a clear example of the relationship between CP and non-superstability. We notice that Mazari-Armida's proof uses very specific module-theoretic tools, most notably the notion of cotorsion-envelope of an $R$-module. For this reason, his proof does not generalise to the abstract setting of an arbitrary variety of algebras. However, we will now see that our Theorem~\ref{main_th1} covers and strengthens a large portion of Mazari-Armida's result. In order to introduce our result we need the following definition.

	\begin{definition}\label{ring__def_perfect} \label{ring_fact} We say that a ring $R$ is \emph{weakly left-perfect} if for any sequence $(r_n : n < \omega) \in R^\omega$ such that, for every $m < \omega$, $r_0 \cdots r_m$ is not a zero-divisor, we have that the sequence of principal ideals $\{ r_0 \cdots r_mR : m < \omega\}$ eventually stabilises.
\end{definition}

Recall that a ring $R$ is \emph{left-perfect} if the condition from \ref{ring__def_perfect} holds without the extra assumption on $r_0 \cdots r_m$ not being zero-divisors. We will see in Section~\ref{sec_modules} that if $R$ is not weakly left-perfect, then $\mathbf{V}_R$ satisfies RCP, thus deducing the following.

\begin{corollary}\label{main_th} Let $R$ be a ring. If $R$ is not weakly left-perfect, then no strong enough  $\mrm{AEC}$-covering of $(\mathcal{F}_{\mathbf{V}_R}, \fleq)$ is superstable.
\end{corollary}

In other terms, by assuming a little bit more than ``not left-perfect", we get a much stronger result than Mazari-Armida's \ref{Armida_fact_nonsuper}, that is, not only $(\mathcal{Flat}_R, \leq_{\mrm{pp}})$ is not superstable but {\em every (strong enough) AEC-covering} of $(\mathcal{F}_{\mathbf{V}_R}, \fleq)$ is not superstable. We stress that by this we mean that in the covering $(\mathcal{K}, \preccurlyeq)$ of ${(\mathcal{F}_{\mathbf{V}_R}, \fleq)}$ we are allowed to change not only the class of models (i.e., from $\mathcal{Flat}_R$ to $\mathcal{K}$) but also the strong submodel relation (i.e., from $\leq_{\mrm{pp}}$ to $\preccurlyeq$) as long as $\preccurlyeq$ is ``strong enough'' in the sense of \ref{def_enough_strong}. For example, it is easy to see that asking that $\preccurlyeq$  implies $\leq_{\mrm{pp}}$ while refining (i.e., being weaker than) direct summands suffices  to fulfill the conditions from \ref{def_enough_strong}.

\smallskip \noindent
We notice that the assumption ``strong enough'' in \ref{main_th} is necessary, since by \cite[Corollary~5.13]{Armida_ros} for any Noetherian ring $R$ we have that $(R\text{-Mod}, \leq)$ is superstable. Finally, concerning the notion of \emph{weakly} left-perfect ring (compared to the notion of left-perfect ring), we recall that Cohn's construction from \cite{Cohn} shows that any commutative ring can be embedded in a ring where every element of the extended ring is either a zero-divisor or an invertible element of the smaller ring.

\subsubsection{Groups}\label{subsec2_intro}

	In \cite{mekler2} Mekler settled to a large extent the problem of which varieties of groups $\mathbf{V}$ satisfy CP (cf. also \cite{carolillo_paolini} on this). In particular, it is shown in \cite{mekler2} that all torsion-free varieties of groups satisfy CP. Relying on this we deduce from our Theorem~\ref{main_th1} the following corollary. Easily, the corollary below covers the two key examples of the variety of groups and of the variety of abelian groups. In the following corollary, we write $F_{\mathbf{V}}(X)$ for the free algebra in the variety $\mathbf{V}$ freely generated by the set $X$.

	\begin{corollary}\label{cor} Let $\mathbf{V}$ be a torsion-free variety of groups such that that if $F_{\mathbf{V}}(X) \in \free$, $x \in X$, $y \in F_{\mathbf{V}}(X)$, $0 < n < \omega$, then $(x^n = y^n \Rightarrow x = y)$. Then $\mathbf{V}$ satisfies $\mrm{RCP}$ and so no strong enough  $\mrm{AEC}$-covering of $(\mathcal{F}_{\mathbf{V}_R}, \fleq)$ is superstable.
\end{corollary}

\subsection{The second venue}\label{intro:second_venue}

So far, our approach was based on the identification of a stronger version of CP, which we called the Reinforced Construction Principle (RCP). In the second part of our work we show that, assuming some conditions on the variety $\mathbf{V}$ and its covering $\mrm{AEC}$, it is possible to derive a non-superstability result already from CP alone. In particular, we proceed in Section~\ref{sec:second_venue} by examining  not the number of types (the superstability notion from Theorem~\ref{main_th1}) but rather the behavior of possible {\em independence calculi} on an AEC-covering $(\mathcal{K}, \preccurlyeq)$. For details on the notion of independence calculus used in the present paper we refer the reader to Section~\ref{sec_nonforking}, but we notice here that this is simply a direct generalisation of the properties satisfied by non-forking in models of a stable first-order complete theory.

\begin{theorem}\label{main_th2} Let $\mathbf{V}$ be a countable variety which satisfies the Construction Principle, and let $(\mathcal{K}, \preccurlyeq)$ be an $\mrm{AEC}$-covering of $(\free, \fleq)$ admitting a weak independence calculus $(\mathcal{K}, \preccurlyeq, \pureindep)$ (cf.~\ref{def:good_frame}). If $(\mathcal{K}, \preccurlyeq, \pureindep)$ is nice with respect to free factors (cf.~\ref{def:good_completion+}), then the calculus $(\mathcal{K}, \preccurlyeq, \pureindep)$ is not superstable (cf.~\ref{superstable_forking}).
\end{theorem}

Also in this case some assumptions are necessary, as we need to rule out cases such as the AEC of abelian groups with the substructure relation, which, as mentioned, is superstable. In particular, we require a form of coherence between the independence calculus $(\mathcal{K}, \preccurlyeq, \pureindep)$ and the relation $\fleq$ of being a free factor with a free complementary factor (cf.~the beginning of the introduction). Crucially, Theorem~\ref{main_th2} exhibits an interesting case where CP is already enough to entail the non-superstability of the covering $\mrm{AECs}$. In particular, the conditions from \ref{def:good_completion+} are satisfied in the key example of the variety of groups; which leads us to the following.

	\begin{definition}\label{def_like_free_groups} Let $\mathbf{V}$ be a variety and let $T$ be the first-order theory of $\freeinf$. We say that $T$ {\em behaves like the theory of free groups} if:
	\begin{enumerate}[(1)]
	\item there is $n_* < \omega$ such that for every $n_* \leq m < \omega$ we have $F_{\mathbf{V}}(m) \in \mrm{Mod}(T)$;
	\item for all $A,B\in \freefg\cap\mrm{Mod}(T)$ such that $A\fleq B$ and $b\in B$, the smallest free factor $D\in \mrm{Mod}(T)$ of $B$ containing $Ab$ belongs to $\mrm{acl}(Ab)$;
	\item $A, B \in \free^{< \infty}$, $C\in \free$ and $A \leq B \leq C$, $A \fleq C$, $B \fleq C$, then $A \fleq B$;
	\item for all $A,B\in \free\cap \mrm{Mod}(T)$, $A\fleq B$ entails $A\preccurlyeq_{\omega, \omega} B$ and, if moreover $A,B\in \freefg$, then $A\preccurlyeq B$ entails $A\fleq B$.
\end{enumerate}
\end{definition}

	As apparent from \ref{def_like_free_groups}(1), the underlying idea of this second approach is the possibility of ``approximating'' free objects of infinite rank by their finite-rank submodels. Crucially, this is possible in the case of free groups thanks to the seminal work of Kharlampovich-Myasnikov \cite{KM06} and Sela \cite{zsela}. From this, and by using some key results from the literature, we show in Section~\ref{sec:second_venue} that free groups in fact satisfy all the assumptions from \ref{def_like_free_groups}. This provides us with yet another proof of the non-superstability of the free groups of infinite ranks;  more generally we have:

	\begin{corollary} Let $\mathbf{V}$ be a variety with CP and let $T$ be the first-order theory of $\freeinf$. If $T$ behaves like the theory of free groups (cf.~\ref{def_like_free_groups}), then $T$ is not superstable.
	\end{corollary}
 
\section{Preliminaries}\label{sec_nonforking}

We recall in this section some background notion on abstract elementary classes $(\mrm{AECs})$ (cf. also \cite{Sh2}). We assume the reader is familiar with basic definitions and constructions from universal algebra and model theory, for which we refer to \cite{burris,tent_book}.

\begin{definition}\label{def_AC} \label{def_AEC} 
	Let $\mathcal{K}$ be a class of $L$-structures and let $\preccurlyeq$ be a binary relation on $\mathcal{K}$. We say that $(\mathcal{K}, \preccurlyeq)$ is an {\em abstract elementary class} $(\mrm{AEC})$ if the following conditions are satisfied:
	\begin{enumerate}[(1)]
		\item  $\mathcal{K}$ and $\preccurlyeq$ are closed under isomorphisms, i.e., if $A\in \mathcal{K}$ and $f:A\to B$ is an isomorphism then $B\in \mathcal{K}$ and if $C\in \mathcal{K}$ is such that $C\preccurlyeq A$ then $f(C)\preccurlyeq B$.
		\item If $A \preccurlyeq B$, then $A$ is an $L$-submodel of $B$ (written $A \leq B$).
		\item The relation $\preccurlyeq$ is a partial order on $\mathcal{K}$.
		\item If $(A_i)_{i < \delta}$ is an increasing continuous $\preccurlyeq$-chain of structures in $\mathcal{K}$, then:
		\begin{enumerate}[({4.}1)]
			\item $\bigcup_{i < \delta} A_i \in \mathcal{K}$;
			\item for each $j < \delta$, $A_j \preccurlyeq \bigcup_{i < \delta} A_i$;
			\item if $A_i \preccurlyeq B$ for all $i<\delta$, then $\bigcup_{i < \delta} A_i \preccurlyeq B$.
		\end{enumerate}
		\item If $A, B, C \in \mathcal{K}$, $A \preccurlyeq C$, $B \preccurlyeq C$ and $A \leq B$, then $A \preccurlyeq B$; 
		\item[(6)] There is a L\"owenheim-Skolem number $\mathrm{LS}(\mathcal{K}, \preccurlyeq)\geq |L|+\aleph_0$ such that if $A \in \mathcal{K}$ and $B \subseteq A$, then there is $C \in \mathcal{K}$ with $B \subseteq C\preccurlyeq A$ and $|C| \leq |B| + \mathrm{LS}(\mathcal{K}, \preccurlyeq)$.
	\end{enumerate}
When $(\mathcal{K},\preccurlyeq)$ is an $\mrm{AEC}$, we refer to $\preccurlyeq$ as a \emph{strong submodel relation} for the class $\mathcal{K}$. We refer to Condition~(4) as the \emph{Tarski-Vaught Axioms}, to Condition~(4.3) specifically as the \emph{Smoothness Axiom} of $\mrm{AEC}$, to Condition~(5) as the \emph{Coherence Axiom}, and to Condition~(6) as the \emph{L\"owenheim-Skolem-Tarski Axiom}.
\end{definition}

\begin{notation}
	Given two $L$-structures $A$ and $B$ we write $A\leq B$ if $A$ is an $L$-substructure of $B$. Also, given an $\mrm{AEC}$ $\mathcal{K}$ and $A,B\in \mathcal{K}$, we say that an embedding $f:A\to B$ is a \emph{strong embedding} if $f(A)\preccurlyeq B$.
\end{notation}

We next introduce the notion of \emph{(weak) independence calculus}, and its background definitions. We notice the following approach is analogous to the one of so-called \emph{good frames} often taken in the $\mrm{AEC}$ literature (cf.~\cite{grossberg2,categoricity,boney_vasey,jarden}). However, given our intended applications, we do not need the larger generality of good frames and introduce the simpler notion of weak independence calculus in Def.~\ref{def:good_frame}.

\begin{definition}\label{def:homogeneity}
	Let $(\mathcal{K},\preccurlyeq)$ be an $\mrm{AEC}$. We say that a model $B\in \mathcal{K}$ is \emph{($\kappa$,$\preccurlyeq$)-universal} if for every $C\in \mathcal{K}$ with $|C|\leq \kappa$ there is a strong embedding $f:C\to B$. We say that $B\in \mathcal{K}$ is \emph{($\kappa$,$\preccurlyeq$)-homogeneous} if for any two strong substructures $A_0\preccurlyeq B$ and $A_1\preccurlyeq B$ of size $<\kappa$ and every isomorphism $f:A_0\to A_1$, there is an automorphism $\pi:B\to B$ such that $\pi\restriction A_0=f$.
\end{definition}

 We refer the reader to \cite{Sh2} for a proof of the following fact and the (standard) definitions of AP and JEP.

\begin{fact}\label{monster_model_fact}
	Let $(\mathcal{K},\preccurlyeq)$ be an AEC with the amalgamation property (AP) and the joint embedding property (JEP), then for all $\kappa\geq \mathrm{LS}(\mathcal{K}, \preccurlyeq)$ there is a model $M\in \mathcal{K}$ which is both ($\kappa$,$\preccurlyeq$)-homogeneous and ($\kappa$,$\preccurlyeq$)-universal.
\end{fact}

\begin{notation}\label{monster_model}
In the context of \ref{monster_model_fact}, letting $\mathfrak{K} = (\mathcal{K},\preccurlyeq)$, we denote by $\mathfrak{M}_{\mathfrak{K}}$ the monster model of $\mathcal{K}$, i.e., a model in $\mathcal{K}$ which is ($\kappa$,$\preccurlyeq$)-homogeneous and ($\kappa$,$\preccurlyeq$)-universal for some large enough cardinal $\kappa$. As standard in the context of AECs, we define types as \emph{Galois types} in the monster model, i.e., as orbits of its automorphism group. Given an element $a\in \mathfrak{M}_{\mathfrak{K}}$ and a strong submodel $B\preccurlyeq \mathfrak{M}_{\mathfrak{K}}$ (of size $<\kappa$, where  $\kappa$ is the size for which $\mathfrak{M}_{\mathfrak{K}}$ is both ($\kappa$,$\preccurlyeq$)-homogeneous and ($\kappa$,$\preccurlyeq$)-universal) we denote by $\mathrm{tp}(a/A)$ the Galois type of $a$ in $\mathfrak{M}$ over $A$, namely the stabiliser subgroup of the automorphisms of $\mathfrak{M}_{\mathfrak{K}}$ which fix pointwise $B$. Given $B\preccurlyeq \mathfrak{M}_{\mathfrak{K}}$, we write $S(B)$ for the set of all Galois types over $B$ in $\mathfrak{M}_{\mathfrak{K}}$. When the class $\mathfrak{K}$ is clear from the context, then we simply write $\mathfrak{M}$ instead of $\mathfrak{M}_{\mathfrak{K}}$.
\end{notation}

As we mentioned above, the following definition of weak independence calculus aims to generalise to the asbtract setting of $\mrm{AECs}$ the property of non-forking in stable first-order theories. Essentially, the properties below in \ref{def:good_frame} are those of \emph{good frames}, with the exception of local character, that we separately consider in \ref{superstable_forking} as characterising the \emph{superstable independence calculi}.

\begin{definition}\label{def:good_frame} Let $(\mathcal{K}, \preccurlyeq)$ be an $\mrm{AEC}$ with a monster model $\mathfrak{M}$. We say that $(\mathcal{K}, \preccurlyeq, \pureindep)$ is a \emph{weak independence calculus} if $\pureindep$ is a ternary relation between tuples $\bar{a}\in \mathfrak{M}^{<\omega}$ and strong submodels of $\mathfrak{M}$ which satisfy the following properties:
	\begin{enumerate}[(i)]
		\item Invariance:  if $\bar{a} \pureindep[A_0] A_1$ and $f \in \mathrm{Aut}(\mathfrak{M})$, then $f(\bar{a}) \pureindep[f(A_0)] f(A_1)$.
		\item Disjointness:  if $\bar{a} \pureindep[A_0] A_1$ then $\bar{a}\in A_1$ if and only if $\bar{a}\in A_0$.
		\item Symmetry: if $A_0 \preccurlyeq A_1 \preccurlyeq \mathfrak{M}$, $\bar{b} \pureindep[A_0] A_1$ and $\bar{a} \in A_1^{<\omega}$, then there exists $A_0 \preccurlyeq A_2$ such that $\bar{b} \in A_2^{<\omega}$ and $\bar{a} \pureindep[A_0] A_2$.
		\item Monotonicity: if $A \preccurlyeq A' \preccurlyeq B' \preccurlyeq B \preccurlyeq \mathfrak{M}$, and $\bar{a} \pureindep[A] B$, then $\bar{a} \pureindep[A'] B'$.
		\item Existence: if $A \preccurlyeq B \preccurlyeq \mathfrak{M}$ and $\bar{a} \in \mathfrak{M}^{< \omega}$, then there exists a $q \in S(B)$ such that $q \restriction A = \mathrm{tp}(\bar{a}/A)$ and for some $\bar{b} \models q$ we have that $b \pureindep[A] B$.
		\item Uniqueness: if $A \preccurlyeq B \preccurlyeq \mathfrak{M}$, $\bar{a} \pureindep[A] B$, $\bar{b} \pureindep[A] B$ and $\mathrm{tp}(\bar{a}/A) = \mathrm{tp}(\bar{b}/A)$, then $\mathrm{tp}(\bar{a}/B) = \mathrm{tp}(\bar{b}/B)$.
		\item Continuity: if $\delta$ is limit, $(A_i)_{i \leq \delta}$ is an continuous $\preccurlyeq$-chain, $A_{\delta} \preccurlyeq \mathfrak{M}$ and $\bar{a}$, $(\bar{a}_i)_{i < \delta}$ are such that $\bar{a}_i \pureindep[A_0] A_i$ and $\mathrm{tp}(\bar{a}_i/A_i) = \mathrm{tp}(\bar{a}/A_i)$, then $\bar{a} \pureindep[A_0] A_{\delta}$.
	\end{enumerate}
\end{definition}

\begin{definition}\label{superstable_forking}
In the context of \ref{def:good_frame} we say that $(\mathcal{K}, \preccurlyeq, \pureindep)$ is a \emph{superstable independence calculus} if in addition the following condition is satisfied:
	\begin{enumerate}[(i)]
	\item[(viii)] Superstable Local Character: if $\delta$ is limit, $(A_i)_{i \leq \delta}$ is an increasing continuous $\preccurlyeq$-chain and $A_{\delta} \preccurlyeq \mathfrak{M}$, then there exists $\alpha < \delta$ such that $\bar{a} \pureindep[A_{\alpha}] A_{\delta}$.
\end{enumerate}
\end{definition}

\section{The first venue}

We consider in this section the first research venue explained in \ref{intro:first_venue}. Towards proving Theorem~\ref{main_th1}, we introduced the Reinforced Construction Principle $(\mrm{RCP})$ (cf.~\ref{RCP_intro}), which is a strengthening of the original construction principle by Eklof and Mekler which holds in most concrete cases.

\begin{notation} Let $B$ be a model and $C \subseteq B$. We denote by $\mrm{dcl}^{\mrm{qf}}_B(C)$ the quantifier-free definable closure of $C$ in $B$.
\end{notation}

\begin{defn}[\ref{RCP_intro}]
	Let $\mathbf{V}$ be a variety of algebras. We say that $\mathbf{V}$ satisfies the Reinforced Construction Principle $(\mrm{RCP})$ if the following conditions are satisfied:
	\begin{enumerate}
		\item[$ (\mrm{RCP}) $] there are countable $A, B \in \freeinf$ with $A \leq B$ such that:
		\begin{enumerate}[(a)]
			\item $A$ is freely generated by $\{a_i : i < \omega\}$, and for every $i < \omega$, letting $A_i = \langle a_j : j \leq i \rangle_B$, we have that $A_i \fleq B_i$;
			\item $A \not \fleq B \ast C$, for every $C \in \free$;
			\item $B$ is freely generated by $\{b_i : i < \omega\}$ and $\mrm{dcl}^{\mrm{qf}}_B(Ab_0) = B$.
		\end{enumerate}
	\end{enumerate}
\end{defn}

	\begin{definition}\label{def_enough_strong} Let $\mathbf{V}$ be a variety of algebras which satisfies the Construction Principle. We say that the covering $(\mathcal{K}, \preccurlyeq)$ of $(\free, \fleq)$ is {\em strong enough} whenever:
	\begin{enumerate}[(i)]
	\item $A, B, C \in \freeinf$ and $A \preccurlyeq B$ implies that $A \ast C \preccurlyeq B \ast C$;
	\item if $A$ and $B$ are as in $\mrm{RCP}$ and $B \preccurlyeq M \in \mathcal{K}$, then $\mrm{dcl}^{\mrm{qf}}_B(Ab_0) = \mrm{dcl}^{\mrm{qf}}_M(Ab_0)$;
	\item if $A$ and $B$ are as in $\mrm{RCP}$ and $f$ is an automorphism of the monster model $\mathfrak{M}$ of $(\mathcal{K}, \preccurlyeq)$ fixing $A$ pointwise, then we have that $\mrm{dcl}^{\mrm{qf}}_{\mathfrak{M}}(Ab_0) = \mrm{dcl}^{\mrm{qf}}_{\mathfrak{M}}(Af(b_0))$.
\end{enumerate}
\end{definition}

	\begin{context} In this section we assume that $\mathbf{V}$ satisfies the assumptions from \ref{RCP_proof}, so in particular, $A, A_i, B$ etc. are as there. We also assume that $(\mathcal{K}, \preccurlyeq)$ is a {\em strong enough} AEC-covering of $(\free, \fleq)$, where the latter is as in \ref{def_enough_strong}.
\end{context}

	\begin{construction}\label{construction:first_venue} Let $\lambda \geq \mrm{LS}(\mathcal{K}, \preccurlyeq)$ be an infinite cardinal and let $A \preccurlyeq \mathfrak{M}$ be the $\mathbf{V}$-free algebra freely generated by $\{a_i : i < \lambda\} \cup \{c_i : i < \omega\}$. 	For every $X \in \lambda^{[\aleph_0]}$, let $B_X, A_X, A'_X, B'_X, B^X$ be such that the following happens:
	\begin{enumerate}[(1)]
	\item $B_X \in \freeinf$ and $B_X$ is freely generated by $\{b^X_i : i < \omega\}$;
	\item we have that $A_X := \langle a_{i} : i \in X \rangle_A \leq B_X$;
	\item the pair $A_X, B_X$ is as in \ref{RCP_proof}, and in particular $\mrm{dcl}^{\mrm{qf}}_{B_X}(A_X b^X_0) = B_X$;
	\item $B'_X \in \freeinf$ is freely generated by $\{b^X_i : i < \omega\} \cup \{c_i : i < \omega\}$;
	\item $A'_X \in \freeinf$ is freely generated by $\{a_i: i \in X\} \cup \{c_i : i < \omega\}$;
	\item $B^X \in \freeinf$ is freely generated by $\{a_i : i < \lambda, i \notin X\} \cup \{b^X_i : i < \omega\} \cup \{c_i : i < \omega\}$;
	\item if $Y \in \lambda^{[\aleph_0]}$, then $\langle B'_XB^Y\rangle_{\mathfrak{M}} = B'_X\ast_{A'_{X \cap Y}}B^Y$, where $A'_{X \cap Y}$ is freely generated by $  \{a_i : i\in X\cap Y\} $ and $ \{c_i : i<\omega \}  $;
	\item $B^X \preccurlyeq \mathfrak{M}$.
\end{enumerate}
\end{construction}

	\begin{claim}\label{claim_proof_1+} For every $X \in \lambda^{[\aleph_0]}$, $A \preccurlyeq B^X \in \mathcal{K}$.
\end{claim}

	\begin{proof} The fact that $B^X \in \mathcal{K}$ is clear, since $B^X \in \freeinf \subseteq \mathcal{K}$. Concerning the fact that $A \preccurlyeq B^X$, it suffices to show that $A'_X \preccurlyeq B'_X$, since then by \ref{def_enough_strong}(i) it follows that $A \preccurlyeq B^X$. Let $X=\{\alpha_i : i<\omega\}$ be an enumeration of $X$, then $A'_X \preccurlyeq B'_X$ is clear, since, for every $i < \omega$, we have that:
	$$A'_{X, i}  := \langle \{ a_{\alpha_j} : j \leq i \} \cup \{c_k : k < \omega\}\rangle_A \in \freeinf,$$
	$$A'_{X, i} \fleq B'_X \in \freeinf,$$
	because of \ref{construction:first_venue}(3), and thus $\bigcup_{i < \omega} A'_{X, i} = A'_X \preccurlyeq B'_X$, as $(\mathcal{K}, \preccurlyeq)$ is an $\mrm{AEC}$-covering of $(\free, \fleq)$.
	\end{proof}
	
%

	\begin{claim}\label{crucial_claim} If $X, Y \in \lambda^{[\aleph_0]}$, $X\neq Y$ and $|X \cap Y| < \aleph_0$, then in $(\mathcal{K}, \preccurlyeq)$ the Galois type of $b^X_0$ over $A$ is different from the Galois type of $b^Y_0$ over $A$.
\end{claim}

	\begin{proof} By \ref{construction:first_venue}(8) we have that $B^X, B^Y \preccurlyeq \mathfrak{M}$. Recall that by \ref{claim_proof_1+} we have that $A \preccurlyeq B^X, B^Y$. Suppose that there is an automorphism $f$ of the monster model $\mathfrak{M}$ of $(\mathcal{K}, \preccurlyeq)$ that fixes $A$ pointwise and sends $b^Y_0$ to $b^X_0$. Then clearly, recalling that $\mrm{dcl}^{\mrm{qf}}_{B_Y}(A_X b^Y_0) = B_Y$, we also have that $\mrm{dcl}^{\mrm{qf}}_{B^Y}(A b^Y_0) = B^Y$. Since we are assuming that the covering $(\mathcal{K}, \preccurlyeq)$  of $(\free, \fleq)$ is as in \ref{def_enough_strong}, $f$ maps $\mrm{dcl}_{\mathfrak{M}}(Ab^Y_0) = \mrm{dcl}_{B^Y}(Ab^Y_0) = B^Y$ to $\mrm{dcl}_{\mathfrak{M}}(Ab^X_0) = \mrm{dcl}_{B^X}(Ab^X_0) = B^X$.  
		
	\smallskip\noindent Notice now that, since $|X \cap Y| < \aleph_0$, it follows from  \ref{construction:first_venue}(7) that $\langle B'_XB^Y\rangle_{\mathfrak{M}} = B'_X\ast_{A'_{X \cap Y}}B^Y$. It clearly follows that $D\coloneqq \langle A'_XB^Y\rangle_{\mathfrak{M}} = A'_X\ast_{A'_{X \cap Y}}B^Y$ and thus $A'_X \fleq A'_X \ast_{A'_{X\cap Y}} B^Y $. Then, we have that $A'_X, B^Y \fleq \langle A'_XB^Y\rangle_{\mathfrak{M}}$ and so, by applying the automorphism $f$, we have that $f(A'_X), f(B^Y) \fleq f(D)\in\free$.  But by the assumptions from \ref{def_enough_strong} this means that $f(A'_X) = A'_X$ and $f(B^Y) = B^X$, and so $A'_X, B^X \fleq f(D)$. But this means in particular that $A_X \fleq B_X \ast C$, for some $C \in \free$, thus contradicting \ref{construction:first_venue}(3).
\end{proof}	

	\begin{thmn}[\ref{main_th1}] Let $\mathbf{V}$ be a countable variety of algebras which satisfies $\mrm{(RCP)}$ (cf.~\ref{RCP_proof}) and let $(\mathcal{K}, \preccurlyeq)$ be a strong enough $\mrm{AEC}$-covering of $(\free, \fleq)$ (cf.~\ref{def_enough_strong}). Then $(\mathcal{K}, \preccurlyeq)$ is not superstable.
\end{thmn}	
\begin{proof} This follows from Construction~\ref{construction:first_venue} and Claim~\ref{crucial_claim}, and by noticing that for an infinite cardinal $\lambda$ such that $\lambda^{\aleph_0} > \lambda$ we can find  a collection $\mathcal{X}$ of subsets of $\lambda^{[\aleph_0]}$ witnessing that $\lambda^{\aleph_0} > \lambda$ and such that for every $X, Y \in \mathcal{X}$ we have $|X \cap Y| < \aleph_0$.
\end{proof}

\subsection{Applications to $R$-modules}\label{sec_modules}

As we explained in the introduction, we next use Theorem~\ref{main_th1} to give a sufficient criterion for an $\mrm{AEC}$ of $R$-modules not to be superstable. Under the Condition~\ref{ring__def_perfect} of not being weakly left-perfect, this provides a generalisation of Fact~\ref{Armida_fact_nonsuper} from \cite{armida3} beyond the setting of $\mrm{AECs}$ of $R$-modules with the relation of pure embeddability.

	\begin{observation}\label{zero_divisors_assumption} Let $F = F(X)$ be a free $R$-module freely generated by $X$. Let $0 \neq a \in F$ and suppose that $ra = 0$. Without loss of generality, we have that $a = \sum_{i < n} s_i x_i$ with $0 \neq s_i \in R$, for all $i < n$ with $0 < n < \omega$. Then we have that:
$$ra = \sum_{i < n} rs_i x_i = 0,$$
but since $F$ is free this means that for all $i < n$ we have that $rs_i = 0$, so necessarily $r$ is a left annihilator of $s_i$ for all $i < n$. Thus, if $r$ is not a left zero-divisor of $R$ (i.e., $\nexists s \in R$, $rs = 0$), then $rx = 0$ has only the trivial solution, that is $x = 0$.

\smallskip 
\noindent	
	Suppose now that $r$ is not a left zero-divisor of $R$. Then
	$F(X) \models \exists! y (ry = r x_i)$.
To see this, let $b \in F$ be a solution, then $rb - rx_i = 0$ and so the above applies, i.e., we have that $rb - rx_i = r(b-x_i)$ and $b-x_i = 0$, i.e., $b = x_i$.
\end{observation}

\begin{defn}[\ref{ring__def_perfect}] We say that a ring $R$ is \emph{weakly left-perfect} if for any sequence $(r_n : n < \omega) \in R^\omega$ such that, for every $m < \omega$, $r_0 \cdots r_m$ is not a zero-divisor we have that the sequence of principal ideals $\{ r_0 \cdots r_mR : m < \omega\}$ eventually stabilises.
\end{defn}

	\begin{proposition}\label{rings:RCP} Let $R$ be a non weakly left-perfect ring. Then $R$ satisfies the Reinforced Construction Principle.
\end{proposition}
	\begin{proof} Let $(r_n : n < \omega) \in R^\omega$ be a witness of non weak left-perfectness, i.e., for all $m < \omega$, we have that $r_0 \cdots r_m$ is not a zero divisor of $R$ and $r_0 \cdots r_mR \supsetneq r_0 \cdots r_{m+1}R$. Let $B$ be freely generated by $\{b_i : i < \omega\}$ and, for $i < \omega$, let $a_i = b_i - r_{i}b_{i+1}$. Let $A = \langle a_i : i < \omega \rangle_B$. Then $A \leq B$ are as wanted. To see that conditions (a) and (b) of \ref{RCP_proof} are satisfied it suffices to recall the proof of \cite[1.1, p.~193]{EM2}. To see that conditions (c) of \ref{RCP_proof} is satisfied it suffices to use Observation~\ref{zero_divisors_assumption}.
	\end{proof}

	\begin{corl}[\ref{main_th}] Let $R$ be a ring and let $\mathbf{V}_R$ be the variety of all $R$-modules. If $R$ is not weakly left-perfect, then no strong enough  $\mrm{AEC}$-covering of $(\mathcal{F}_{\mathbf{V}_R}, \fleq)$ is superstable.
	\end{corl}
	\begin{proof}
	By Proposition~\ref{rings:RCP} and Theorem~\ref{main_th}.
	\end{proof}

\subsection{Applications to groups}\label{sec_groups}

	A second consequence of Theorem~\ref{main_th1} is in the setting of groups. We prove the following corollary and we stress that, in particular, it covers the two key cases of the varieties of groups and the variety of abelian groups.

	\begin{corl}[\ref{cor}] Let $\mathbf{V}$ be a torsion-free variety of groups such that that if $F_{\mathbf{V}}(X) \in \free$, $x \in X$, $y \in F_{\mathbf{V}}(X)$, $0 < n < \omega$, then $(x^n = y^n \Rightarrow x = y)$. Then $\mathbf{V}$ satisfies $\mrm{RCP}$ and so no strong enough  $\mrm{AEC}$-covering of $(\mathcal{F}_{\mathbf{V}_R}, \fleq)$ is superstable.
	\end{corl}
	
	\begin{proof} Let $B$ be the $\mathbf{V}$-free group freely generated by $\{b_i : i < \omega\}$. For $i < \omega$, let $a_i = b_ib^{-2}_{i+1}$ and let $A = \langle a_i : i < \omega \rangle_B$, then by \cite[p.~131]{mekler2} (cf.~also \cite[Ex.~4.5]{carolillo_paolini}) we have that the pair $(A, B)$ is a witness of CP. If in addition the condition in the statement of the corollary is satisfied, then obviously also condition \ref{RCP_intro}(c) is satisfied and so RCP holds. Therefore, it follows from Theorem~\ref{main_th} that no strong enough  $\mrm{AEC}$-covering of $(\mathcal{F}_{\mathbf{V}_R}, \fleq)$ is superstable.
\end{proof}

\begin{remark}
	Let $\mathbf{V}$ be any (countable) variety of algebras, and let $T$ be the first-order theory of the $\mathbf{V}$-free algebras, then $(\mrm{Mod}(T),\preccurlyeq_{\omega,\omega})$ is clearly a strong enough covering of  $(\mathcal{F}_{\mathbf{V}}, \fleq)$ in the sense of \ref{def_enough_strong}. Therefore, Corollary~\ref{main_th} and Corollary~\ref{cor} show  that $\mathbf{V}$ is not superstable, whenever $\mathbf{V}$ is the variety of $R$-modules for $R$ a non weakly left-perfect ring, or $\mathbf{V}$ is a torsion-free variety of groups such that that if $F_{\mathbf{V}}(X) \in \free$, $x \in X$, $y \in F_{\mathbf{V}}(X)$ and $0 < n < \omega$, then the equation $x^n = y^n$ implies that $x = y$.
\end{remark}

\section{The second venue}\label{sec:second_venue}
 
We consider now a second proof of non-superstability under the assumption of the satisfaction of CP for the variety $\mathbf{V}$. As explained in \ref{intro:second_venue}, we do not exhibit here $\lambda^{\aleph_0}$-many types over a structure of size $\lambda$, but rather we approach the problem from the viewpoint of the existence of an independence calculus (in the sense of \ref{def:good_frame}). Admittedly, the conditions from \ref{def:good_completion+} below are rather strong, but they display the fact that, in some circumstances, the original  $\mrm{CP}$ by Eklof and Mekler (and not $\mrm{RCP}$) already entails non-superstability. 

\begin{definition}\label{def:ast_independence}
	Suppose $(\mathcal{F}_{\mathbf{V}},\fleq)$ has an $\mrm{AEC}$-covering $\mathfrak{K} = (\mathcal{K},\preccurlyeq)$ which admits a weak independence calculus $(\mathcal{K},\preccurlyeq, \pureindep)$ (cf.~\ref{def:good_frame}), and let $A,B,C\in \mathcal{F}_{\mathbf{V}}$ with $A,B,C\preccurlyeq \mathfrak{M}_\mathfrak{K}$, $A\fleq B, C$ then we let 
	\begin{align*}
		B \ind^\ast_A C \Longleftrightarrow B\ast_AC=\langle B C\rangle_{\mathfrak{M}_{\mathfrak{K}}} \preccurlyeq \mathfrak{M}_\mathfrak{K}.
	\end{align*}
\end{definition}

\begin{definition}\label{def:good_completion+} Suppose $(\mathcal{K}, \preccurlyeq)$ is an AEC-covering of $(\freeinf, \fleq)$ admitting a weak independence calculus (cf. \ref{def:good_frame}) and denote by $\mathfrak{M}$ its monster model. We say that $(\mathcal{K}, \preccurlyeq, \pureindep)$ is {\em nice with respect to free factors} if the following conditions are satisfied:
	\begin{enumerate}[$(\star_1)$]
		\item there is $n_* < \omega$ such that for every $n_* \leq m < \omega$ we have $F_{\mathbf{V}}(m) \in \mathcal{K}$;
		\item there is $m_*<\omega$ such that, if $A,B,C\in \free$ with $A,B,C\preccurlyeq \mathfrak{M}_{\mathfrak{K}}$, $A\fleq B$, $A\fleq C$ and $A$ has rank at least $m_*$, then $B \pureindep[A] C $ if and only if $B \astindep[A] C$;
		\item if $A\in \freefg\cap \mathcal{K}$,  $B,C\in \free\cap \mathcal{K}$ with $A,B,C\preccurlyeq \mathfrak{M}_{\mathfrak{K}}$ and $A\fleq C$, then for every $c\in C$ such that $c\pureindep[A] B $ there is $D\in\freefg$ such that $c\in D$, $A\fleq D\fleq C$ and $D \pureindep[A] B $.
		\item if $A,B,C\in \freefg\cap \mathcal{K}$, $D\in\free\cap \mathcal{K}$ with $A,B,C,D\preccurlyeq \mathfrak{M}_{\mathfrak{K}}$, $A\fleq B\fleq D$, $A\fleq C\fleq D$ and $B\ind_A C$, then it follows that $B\ast_A C\fleq D$.
	\end{enumerate}	
\end{definition} 

\begin{context}\label{context_second_venue} In this section we assume that $\mathbf{V}$ is a countable variety of algebras satisfying $\mrm{CP}$, and that $(\mathcal{K}, \preccurlyeq)$ is an $\mrm{AEC}$-covering of $(\free, \fleq)$ admitting a weak independence calculus $(\mathcal{K}, \preccurlyeq, \pureindep)$ (cf. \ref{def:good_frame}) which is nice with respect to free factors (cf.~\ref{def:good_completion+}). Also, we denote by $\mathfrak{M}$ the monster model of $(\mathcal{K}, \preccurlyeq)$ (cf.~\ref{monster_model}).
\end{context}

\begin{lemma}\label{lemma_for_main_th} If $(A_i : i < \omega)$ and $(B_i : i < \omega)$ are $\fleq$-chains of objects from $\free\cap \mathcal{K}$, and for every $i < \omega$ we have that $A$ is a $\mathbf{V}$-free algebra or rank $\geq m_*$ (recall \ref{def:good_completion+}$(\star_2)$) such that $A_i\fleq B_i$ and $B_i \pureindep[A_i] \bigcup_{j < \omega} A_j$, then $\bigcup_{i < \omega} A_i \fleq \bigcup_{i < \omega} B_i$.
\end{lemma}

\begin{proof}
	By \ref{def:good_completion+}$(\star_2)$, $B_i \pureindep[A_i] \bigcup_{j < \omega} A_j$ implies that $B_i \astindep[A_i] \bigcup_{j < \omega} A_j$, which means that $\langle B_i \cup \bigcup_{j < \omega} A_j \rangle = B_i \ast_{A_i} \bigcup_{j < \omega} A_j \preccurlyeq \mathfrak{M}$.   Let $A=\bigcup_{j < \omega} A_j$, $B=\bigcup_{j < \omega} B_j$ and for every $i<\omega$ let $C_i,D_i\in \mathcal{K}$ be such that $B_{i+1}=B_i\ast C_i$ and $A_{i+1}=A_i\ast D_i$. Also, let $E_i\in \mathcal{K}$ be such that $B_i=A_i\ast E_i$. Then the following holds for all $i<\omega$:
	\[B_i \ast_{A_i} \bigcup_{j < \omega} A_j = E_i \ast A_i \ast \bigast_{i<j<\omega} D_i, \]
	whence
	\[B= (E_i \ast A_i \ast \bigast_{i<j<\omega} D_i)\ast  \bigast_{i<j<\omega} C_i.  \]
	Since $A_i \ast \bigast_{i<j<\omega} D_i=  \bigcup_{j < \omega} A_j$, this immediately entails that $ \bigcup_{j < \omega} A_j\fleq B$.
\end{proof}

\begin{thmn}[\ref{main_th2}]
	Let $\mathbf{V}$ and $(\mathcal{K}, \preccurlyeq, \pureindep)$  be as in Context~\ref{context_second_venue}, then  $(\mathcal{K}, \preccurlyeq, \pureindep)$ is not superstable (cf.~\ref{superstable_forking}).
\end{thmn}
\begin{proof}
	Following Context~\ref{context_second_venue} we let $(\mathcal{K},\preccurlyeq, \pureindep)$ be an $\mrm{AEC}$-covering of $(\free, \fleq)$ admitting a weak independence calculus. Towards contradiction, we assume that $(\mathcal{K},\preccurlyeq, \pureindep)$ is superstable in the sense of \ref{superstable_forking}. Since $\mathbf{V}$ satisfies $\mrm{CP}$, we can find two countable $A,B\in \freeinf$ with $A\leq B$, $A$ freely generated by $\{a_i : i <\omega  \}$ such that $\langle a_j : j\leq i \rangle_B \fleq B $ and $A\not\fleq B\ast C$ for all $C\in \freeinf$. Let $X_C=\{c_i : i<\omega  \}$ be any collection of elements from $C$ so that $\langle AX_C\rangle_X=C$ and, for every $i<\omega$ we let $A_i= \langle a_j : j\leq i \rangle_B$. Notice that, by \ref{def:good_completion+}$(\star_1)$, we have that the chain $(A_i)_{i<\omega}$ eventually belongs to $\mathcal{K}$. 
	
	\medskip 
	\noindent \underline{Claim}. There is an infinite increasing sequence  $I\coloneqq\{i_n : n<\omega \}\subseteq\omega$ such that, for all $i\in I$, there is $C_i\in \freefg$ satisfying the following conditions:
	\begin{enumerate}[(1)]
		\item $c_i\in C_i\fleq C$;
		\item $ C_i\pureindep[A_{i}] A$;
		\item $A_{i_0}\fleq C_0$ and, for all $n<\omega$, $A_{i_{n+1}}\ast_{A_{i_n}} C_n \fleq C_{n+1}$.
	\end{enumerate}
	
	\smallskip
	\noindent\emph{Proof of the claim.} We prove the claim by induction on $n<\omega$.
	
	\smallskip
	\noindent \underline{Case $n=0$}. 
	\newline We start by considering the element $c_0$. Since $(\mathcal{K}, \preccurlyeq,\pureindep)$ is a superstable forking calculus in the sense of \ref{superstable_forking}, it follows that there is some $i_0<\omega$ such that $A_{i_0}\in \mathcal{K}$ and $c_0\pureindep[A_{i_0}] A$. Notice that, using the monotonicity of $(\mathcal{K}, \preccurlyeq,\pureindep)$ (cf.~\ref{def:good_frame}), we can assume without loss of generality that $A_{i_0}$ has rank $\geq n_*,m_*$, where these two values are as in \ref{def:good_completion+}. Since $A_{i_0}$ is finitely generated, we have by \ref{def:good_completion+}$(\star_3)$ that there is some $C_{0}\in \freefg\cap \mathcal{K}$ such that $c_0\in C_{0}$, $A_{i_0}\fleq C_0\fleq C$, and $C_0\pureindep[A_{i_0}] A$, verifying our claim.
	
	\smallskip
	\noindent \underline{Case $n=m+1$}.
	\newline From the induction hypothesis, there is $C_{m}\in \freefg\cap \mathcal{K}$ with  $c_{m}\in C_{m}$, $A_{i_m}\fleq C_{m}\fleq  C$ and $C_{m}\pureindep[A_{i_m}] A$. From the fact that $A_{i_m}\fleq A,C_{m}$ and $A_{i_k}\fleq A_\ell$ whenever $i_k<\ell$, together with $C_{m}\pureindep[A_{i_m}] A$ and \ref{def:good_completion+}$(\star_2)$, it follows that the free amalgams $A_{k}\ast_{A_{i_m}} C_{m}$ and $ A\ast_{A_{i_m}} C_{m}$ are well-defined strong substructure of the monster model for all $i_m\leq k<\omega$. Since the family of structures $(A_{k}\ast_{A_{i_m}} C_{m})_{k<\omega}$ forms a strong $\preccurlyeq$-chain in $(\mathcal{K}, \preccurlyeq)$, it follows again  from \ref{superstable_forking} that there is some $i_m<i_n<\omega$ such that
	\[c_n\pureindep[A_{i_n}\ast_{A_{i_m}} C_m] A\ast_{A_{i_m}} C_m.\]
	Since $A_{i_n}\fleq C$, $C_m\fleq C$, and (by monotonicity) $C_m\ind_{A_{i_m}}A_{i_n}$, it follows from \ref{def:good_completion+}$(\star_4)$ that  $A_{i_n}\ast_{A_{i_m}} C_m\fleq C$. Because $A_{i_n}\ast_{A_{i_m}} C_m$ is clearly finitely generated, we can apply again \ref{def:good_completion+}$(\star_3)$ and obtain some $C_n\in \freefg$ with $c_{n}\in C_n$ such that $A_{i_n}\ast_{A_{i_m}} C_m \fleq C_n\fleq C$ and, moreover:
	\begin{align*}
		C_n\ind_{A_{i_n}\ast_{A_{i_m}} C_m} A\ast_{A_{i_m}} C_{m}. 
	\end{align*}
	By the definition of $C_n$, it is immediately clear that conditions (1) and (3) from the claim are satisfied, whence it remains to verify (2). Therefore, from the last display above and  \ref{def:good_completion+}$(\star_2)$, we can immediately derive that 
	\[C_n\ind^\ast_{A_{i_n}\ast_{A_{i_m}} C_m}A\ast_{A_{i_m}} C_{m}   \]
	which means that 
	\[ C_n \ast_{(A_{i_n}\ast_{A_{i_m}} C_m)} (A\ast_{A_{i_m}} C_{m})  = \langle C_n \cup  (A\ast_{A_{i_m}} C_{m}) \rangle_{\mathfrak{M}_{\mathfrak{K}}}\preccurlyeq \mathfrak{M}_{\mathfrak{K}}.  \]
	In particular, we have that $C_n\cap (A\ast_{A_{i_m}} C_{m})=A_{i_{n}}\ast_{A_{i_m}} C_{m}$. Because $C_{m}\pureindep[A_{i_m}] A$, it follows again by \ref{def:good_completion+}$(\star_2)$  that $C_n\cap A= A_{i_n}$. Then, since we also have  $A_{i_n}\fleq C_n$, it follows that the structure $C_n \ast_{A_{i_n}} A$ is well-defined, and clearly:
	\[ C_n \ast_{A_{i_n}} A = \langle C_n A \rangle_{\mathfrak{M}_{\mathfrak{K}}}\preccurlyeq \mathfrak{M}_{\mathfrak{K}}.  \]
	Therefore, by applying again \ref{def:good_completion+}$(\star_2)$, we conclude that $C_n \pureindep[A_{i_n}] A$, completing the proof of the induction step, and of the claim. \hfill $\dashv$
	
	\medskip
	\noindent Finally, consider the family of structures $(C_i)_{i<\omega}$ defined in the claim. Since $C_i\in \freeinf$ for all $i<\omega$, it follows from item (3) that $(C_i)_{i<\omega}$  is a strong $\fleq$-chain in $\freeinf$. Therefore, by applying Lemma~\ref{lemma_for_main_th}, we obtain that $A=\bigcup_{n<\omega}A_{i_n}\fleq \bigcup_{n<\omega}C_n$. However, it follows from item (1) in the claim that $\bigcup_{i < \omega}C_i=C$, which yields $A\fleq C$ and contradicts the fact that $A,C$ witness $\mrm{CP}$. This concludes the proof.
\end{proof}

\subsection{Applications}

We conclude the paper by showing that, using some results from the literature, Theorem~\ref{main_th2} yields yet another proof of the non-superstability of the first-order theory of non-abelian free groups, and, more generally, of the first-order theory of free objects of infinite rank with similar model-theoretic properties. We recall that, if $T$ is a stable first-order theory, then there is a canonical independence relation $\pureindep$ satisfying the properties from Section~\ref{sec_nonforking}, which is exactly the non-forking independence relation (cf.~\cite[Thm.~8.5.5]{tent_book}). We start by providing a characterisation of the non-forking relation in $(\mrm{Mod}(T),\preccurlyeq_{\omega, \omega})$ between free algebras $A\fleq B,C$ in $\mrm{Mod}(T)$, and under the assumption that $T$ is stable, and superstable, respectively. This in particular will allow us to verify Condition~\ref{def:good_completion+}$(\star_2)$

\begin{remark}\label{remark:kappa(T)}
	If $T$ is a stable theory in a language $L$, then there always exists a cardinal $\kappa(T)\leq |L|^+$ such that, for all $a\in \mathfrak{M}$ and $B\subseteq \mathfrak{M}$, there is a subset $A\subseteq B$ of size $<\kappa(T)$ such that $a\pureindep[A]B$. This property is called the \emph{locality property} of forking independence. In particular, if $L$ is countable, then we always have that $\kappa(T)\leq \aleph_1$, while $\kappa(T)=\aleph_0$ is equivalent to the statement that $T$ is superstable. 
\end{remark}

\begin{proposition}\label{prop:star_1}
	Let $\mathbf{V}$ be a countable variety of algebras, let $T$ be the first-order theory of the $\mathbf{V}$-free algebras of infinite rank and let $\mathfrak{K}=(\mrm{Mod}(T),\preccurlyeq_{\omega,\omega})$. Let $A,B,C\in \mathcal{F}_{\mathbf{V}}$ with $A,B,C\preccurlyeq_{\omega,\omega} \mathfrak{M}_{\mathfrak{K}}$, $A\fleq B$ and $A\fleq C$, then
	\begin{enumerate}[(1)]
		\item if $T$ is stable and $A\in \freeinf$ then $B \pureindep[A] C$ if and only if $B \astindep[A] C$.
		\item if $T$ is superstable and there is some $n_*<\omega$ such that the $\mathbf{V}$-free algebras of rank $\geq n_*$ are models of $T$,  then there is $n_*\leq m_*<\omega$ such that $B \pureindep[A] C$ if and only if $B \astindep[A] C$. 
	\end{enumerate}
\end{proposition}
\begin{proof}
	We first prove (1). Let $X_A=\{a_i : i<\kappa \}$ be a basis of $A$, where $\aleph_0\leq \kappa$, and let $X_C=\{a_i : i<\lambda \}$ be a basis of $C$, where we assume $\kappa\leq \lambda$. We can assume without loss of generality that $B=A\ast \langle X_B\rangle $ for  $X_B=\{ b_i : i\leq n   \}$ for some finite $n<\omega$. We show that  $B \pureindep[A] C$ if and only if $B \astindep[A] C$. 
	
	\smallskip\noindent Suppose first that $B \astindep[A] C$, then $B\ast_A C=\langle X_BX_C\rangle \preccurlyeq_{\omega,\omega} \mathfrak{M}$. Notice that, by the local character of forking independence (and the fact $\mathbf{V}$ is countable), we can find by \ref{remark:kappa(T)} a countable subset $I\subseteq \kappa$ such that 
	\begin{align*}
		X_B\pureindep[\langle a_i : i\in I\rangle] X_C.
	\end{align*}
	Since $B\ast_A C=\langle X_BX_C\rangle \preccurlyeq_{\omega,\omega} \mathfrak{M}$, every permutation $\pi$ of $\lambda$ induces an automorphism $f_\pi$ such that $f_\pi(a_i)=a_{\pi(i)}$ for all $i\in I$ and, moreover $f\restriction X_B$ is the identity. It follows that $X_B\pureindep[\langle f_\pi(a_i) : i\in I\rangle ] X_C$ for all permutations $\pi$ of $\lambda$, which clearly yields $X_B\pureindep[Z] X_C$ for all countable $Z\subseteq A$. By the finite character of forking independence, it follows that $B \pureindep[A] C$. 
	
	\smallskip\noindent The converse direction easily follows from stationarity.  Suppose we have that $B \pureindep[A] C$ and consider the structure $B\ast_A C$. Notice that, since $A,B,C\in \mathcal{F}_{\mathbf{V}}$ we clearly have that $B\ast_A C\in\mathcal{F}_{\mathbf{V}}$. Since $A\preccurlyeq_{\omega,\omega}C\preccurlyeq_{\omega,\omega}B\ast_A C$ and $A,C \preccurlyeq_{\omega,\omega}\mathfrak{M}$, it follows by the universality and homogeneity property of the monster model $\mathfrak{M}$ that we can find a subset $B'\subseteq \mathfrak{M}$ such that $B'\ast_A C\preccurlyeq_{\omega,\omega}\mathfrak{M}$ and $\type(B'/A)=\type(B/A)$. By the direction already proved we obtain that $B' \pureindep[A] C$. Since forking independence is stationary over models, it follows that $\type(B'/C)=\type(B/C)$. Since $B' \astindep[A] C$ we thus obtain that $B \astindep[A] C$.
	
	\medskip
	\noindent We consider (2). The direction from left to right follows again from stationarity. For the direction from right to left we reason as follows. Let $X_A$ be the basis of $A$, $X_C\supseteq X_A$ be a basis of $C$, and $X_B$ be such that $X_AX_B$ is a basis of $B$. We prove the claim by induction on the size of $X_B$.
	
	\smallskip
	\noindent If $X_B=\{b_0\}$ is a singleton, then since $T$ is superstable, we can reason exactly as in the proof of (1) using the fact that $\kappa(T)=\aleph_0$ (recall \ref{remark:kappa(T)}). Then the set of indices $I$ from the previous proof can be taken to be finite, so that letting $m_*\geq |I|$ the claim holds. Now, let $X_B=\{b_0,\dots, b_n,b_{n+1}\}$ and suppose $B\astindep[A]C$, i.e., $B\ast_A C =\langle ACX_B \rangle\preccurlyeq_{\omega, \omega} \mathfrak{M}$. It also follows that $A\ast \langle b_0,\dots, b_n\rangle \astindep[A]C$ and $A\ast \langle X_B\rangle \astindep[A\ast \langle b_0,\dots, b_n\rangle]C\ast \langle b_0,\dots, b_n\rangle$. By the induction hypothesis it follows that $b_0,\dots, b_n\pureindep[A]C$. Similarly, applying the induction hypothesis to the extension $A\ast \langle b_0,\dots, b_n\rangle\fleq A\ast \langle X_B\rangle$, we also obtain that $A\ast \langle X_B\rangle \pureindep[A\ast \langle b_0,\dots, b_n\rangle]C\ast \langle b_0,\dots, b_n\rangle$. By monotonicity we thus obtain that $b_0,\dots, b_n\pureindep[A]C$ and $A\ast \langle X_B\rangle \pureindep[A\ast \langle b_0,\dots, b_n\rangle]C$, so it follows by transitivity that $A\ast \langle X_B\rangle \pureindep[A]C$, completing our proof.
\end{proof}

Motivated by free groups, we introduce the following definition of {\em behaving like the theory of free groups}, and we prove that it entails the conditions from \ref{def:good_completion+}.

\medskip

\begin{defn}[\ref{def_like_free_groups}] Let $\mathbf{V}$ be a variety and let $T$ be the first-order theory of $\freeinf$. We say that $T$ {\em behaves like the theory of free groups} if we have the following:
	\begin{enumerate}[(1)]
	\item there is $n_* < \omega$ such that for every $n_* \leq m < \omega$ we have $F_{\mathbf{V}}(m) \in \mrm{Mod}(T)$;
	\item for all $A,B\in \freefg\cap\mrm{Mod}(T)$ such that $A\fleq B$ and $b\in B$, the smallest free factor $D\in \mrm{Mod}(T)$ of $B$ containing $Ab$ belongs to $\mrm{acl}(Ab)$;
	\item $A, B \in \free^{< \infty}$, $C\in \free$ and $A \leq B \leq C$, $A \fleq C$, $B \fleq C$, then $A \fleq B$;
	\item for all $A,B\in \free\cap \mrm{Mod}(T)$, $A\fleq B$ entails $A\preccurlyeq_{\omega, \omega} B$ and, if moreover $A,B\in \freefg$, then $A\preccurlyeq_{\omega, \omega} B$ entails $A\fleq B$.
\end{enumerate}
\end{defn}

\begin{lemma}\label{lem:free}
	Let $\mathbf{V}$ be a countable variety of algebras and let $T$ be the first-order theory of $\freeinf$. If $T$ behaves like the theory of free groups (cf.~\ref{def_like_free_groups}) then $(\mrm{Mod}(T), \preccurlyeq_{\omega, \omega})$ satisfies the Condition \ref{def:good_completion+}$(\star_1)$, \ref{def:good_completion+}$(\star_3)$ and \ref{def:good_completion+}$(\star_4)$.
\end{lemma}
\begin{proof}
	Condition \ref{def:good_completion+}$(\star_1)$ is exactly \ref{def_like_free_groups}(1), thus we consider Condition \ref{def:good_completion+}$(\star_3)$ and Condition \ref{def:good_completion+}$(\star_4)$.
	
	\smallskip
	\noindent We first verify Condition \ref{def:good_completion+}$(\star_3)$. Let $A,B$ and $C$ be as in the statement of the condition, and let $c\in C$. Since $A\fleq C$, it follows from \ref{def_like_free_groups}(2) that the smallest free factor $D\fleq C$ containing $Ac$  is contained in the algebraic closure $\mrm{acl}(Ac)$, whence by the properties of non-forking it follows that $D\pureindep[A]B$. Finally, since $A\leq D$ and $D\fleq C$, it follows from \ref{def_like_free_groups}(3) that $A\fleq D$ as well.
	
	\smallskip
	\noindent  We verify Condition \ref{def:good_completion+}$(\star_4)$. Let $A,B,C$ and $D$ be as in the statement of the condition. Since $B\pureindep[A]C$, it follows by Condition~\ref{def:good_completion+}$(\star_2)$ that $B\astindep[A]C$ and therefore $B\ast_A C\preccurlyeq_{\omega, \omega} \mathfrak{M}$, whence by Coherence also $B\ast_A C\preccurlyeq_{\omega, \omega} D$. Let $D=\langle X_D\rangle$ for some infinite basis $X_D=\{d_i : i<\omega\}$ and define $D_n=\langle d_i : i\leq n \rangle_D$. Clearly, since $B\ast_A C$ is finitely generated, we have that $B\ast_A C\leq D_n$ for some $n<\omega$. Notice that by construction we have $D_n\fleq D$, which entails by \ref{def_like_free_groups}(4) that $D_n\preccurlyeq_{\omega, \omega} D$. It follows from the Coherence of $\preccurlyeq_{\omega, \omega}$ that $B\ast_A C\preccurlyeq_{\omega, \omega} D_n$. Finally, since $B\ast_A C$ and $D_n$ are finitely generated, we conclude from \ref{def_like_free_groups}(4) that $B\ast_A C \fleq D_n$.	
\end{proof}

As explained in the introduction, the following fact is well-known (and follows also from \ref{main_th1}), but it makes for an interesting non-trivial application of \ref{main_th2}. 

\begin{corollary}
	The theory $T_{\mrm{fg}}$ of non-abelian free groups is not superstable.
\end{corollary}
\begin{proof}
	Towards contradiction, we suppose that $T_{\mrm{fg}}$ is superstable. Then by Proposition~\ref{prop:star_1} it follows that $T_{\mrm{fg}}$ satisfies Condition \ref{def:good_completion+}$(\star_2)$.  By Lemma~\ref{lem:free} and Theorem~\ref{main_th2}, it is enough to verify that $T_{\mrm{fg}}$ satisfies the conditions from Definition~\ref{def_like_free_groups}.
	
	\smallskip \noindent To this end, we simply refer to key results from the literature. First, Condition \ref{def_like_free_groups}(1) follows from  the fundamental results from \cite{KM06,zsela}, free groups of rank $\geq 2$ belong to $\mrm{Mod}(T)$. Condition \ref{def_like_free_groups}(2) follows from \cite[Thm.~3.16]{Miasnikov} and the observation from \cite[p.~2526]{houcine2} that algebraic extensions from \cite{Miasnikov} correspond	essentially to the notion of algebraic closure restricted to quantifier-free formulas. Finally, Condition \ref{def_like_free_groups}(3) follows from \cite[Lem.~2.4]{Miasnikov} and Condition \ref{def_like_free_groups}(4) from  \cite[Thm.~4]{zsela} and \cite[Thm.~1.3]{Perin}
\end{proof}

	\begin{remark}
		We conclude the paper by mentioning another application of Theorem~\ref{main_th2}. Let $R$ be a principal ideal domain (PID) which is not left-perfect, and recall that by this latter fact the free $R$-modules satisfy the Construction Principle (cf.~also \ref{rings:RCP} and \cite[1.1, p.~193]{EM2}). We consider any $\mrm{AEC}$ $(\mathcal{K}, \leq_{\mrm{pp}})$ which contains the free $R$-modules and with the relation of pure submodule $\leq_{\mrm{pp}}$. Condition~\ref{def:good_completion+}$(\star_1)$ is clearly satisfied, and Condition ~\ref{def:good_completion+}$(\star_2)$ can be proved by reasoning similarly as we did in Proposition~\ref{prop:star_1}. Then, Condition~\ref{def:good_completion+}$(\star_3)$ and Condition~\ref{def:good_completion+}$(\star_4)$ can be verified by noticing that $(\mathcal{K}, \leq_{\mrm{pp}})$ exhibits the same properties isolated in \ref{def_like_free_groups}, but with respect to the relation $\leq_{\mrm{pp}}$ of pure submodule. 
		
		\smallskip \noindent In particular, \ref{def_like_free_groups}(2) holds by considering pure subclosures, while \ref{def_like_free_groups}(3) follows by the fact that $R$ is a PID, and thus projective $R$-modules are free. Finally, a version of \ref{def_like_free_groups}(4) holds because, whenever $A,B$ are free $R$-modules, then $A\fleq B$ entails $A\leq_{\mrm{pp}} B$ and, if moreover $A$ and $B$ are finitely generated, then $A\leq_{\mrm{pp}} B$ entails $A\fleq B$. This shows that Theorem~\ref{main_th2} provides yet another proof of Mazari-Armida's non-superstability result from \cite{armida3}, for the restricted setting of modules over non left-perfect PID.
	\end{remark}


\begin{thebibliography}{10}

	\bibitem{categoricity}
	J.~T. Baldwin
	\newblock {\em Categoricity}.
	\newblock Vol. 50. American Mathematical Society, Providence, 2009.
	
	\bibitem{Baldwin}
	J.~T. Baldwin, P. C. Eklof  and J. Trlifaj.
	\newblock {\em $N^\bot$ as an abstract elementary class}.
	\newblock Ann. Pure Appl. Logic {\bf 149.1-3} (2007), 25-39.
	
	\bibitem{Baldwin_Shelah}
	J. T. Baldwin and S. Shelah.
	\newblock {\em Examples of non-locality}.
	\newblock The Journal of Symbolic Logic, {\bf 73} (2008), no. 3, 765-782.
	
	\bibitem{Baudisch}
	A. Baudisch
	\newblock {\em On superstable groups}.
	\newblock J. Lond. Math. Soc. {\bf s2-42} (1990), no. 3, 452-464.
	
	\bibitem{boneyLmodule}
	W. Boney.
	\newblock {\em A Module-theoretic Introduction to Abstract Elementary Classes}.
	\newblock arXiv preprint arXiv:2503.105281.
	
	\bibitem{boney}
	W. Boney and S. Unger.
	\newblock {\em Large cardinal axioms from tameness in $\mrm{AECs}$}.
	\newblock Proc. Am. Math. Soc. {\bf 145}, (2017), No. 10, 4517--4532.
		
	\bibitem{boney_vasey}
	W. Boney and S. Vasey.
	\newblock {\em Tameness and frames revisited}.
	\newblock J. Symb. Log. {\bf 82} (2017), no. 3, 995-1021.
	
	\bibitem{burris}
	S. Burris and P. H. Sankappanavar.
	\newblock {\em A course in universal algebra}.
	\newblock Springer, New York, 1981. 
		
	\bibitem{carolillo_paolini}
	D. Carolillo and G. Paolini
	\newblock {\em The construction principle and non homogeneity of uncountable relatively free groups}.
	\newblock Arch. Math. Log. {\bf 64}, (2025), no. 7, 1181--1195.	
	
	\bibitem{Cohn}
	P. M. Cohn.
	\newblock {\em Rings of zero-divisors}.
	\newblock Proc. Am. Math. Soc. {\bf 9}, (2017), No. 6, 909-914.
		
	\bibitem{dahmani}
	F. Dahmani, V. Guirardel and D. Osin. 
	\newblock {\em 	Hyperbolically embedded subgroups and rotating families in groups acting on hyperbolic spaces}.
	\newblock Mem. Amer. Math. Soc. \textbf{245} (2017), no. 1156.

	\bibitem{EM}
	P. C. Eklof and A. H. Mekler.
	\newblock {\em Categoricity results for $L_{\infty, \kappa}$}. 
	\newblock Ann. Pure Appl. Logic {\bf 37.1} (1988), 81-99.
	
	\bibitem{EM2}
	P. C. Eklof and A. H. Mekler.
	\newblock {\em Almost free modules: Set-theoretic methods}.
	\newblock Elsevier, Amsterdam: North-Holland, 2002. 
	
		\bibitem{grossberg}
	R. Grossberg and S. Vasey,
	\newblock {\em 	Equivalent definitions of superstability in tame abstract elementary classes}.
	\newblock J. Symb. Log. \textbf{82} (2017), no. 4,  1387--1408.
	
	\bibitem{grossberg2}
	R. Grossberg and M. VanDieren,
	\newblock {\em 	Galois-stability for tame abstract elementary classes}.
	\newblock J. Math. Log. \textbf{6} (2016), no. 1,  25--48.
	
	
	\bibitem{houcine}
	A. Ould Houcine.
	\newblock {\em 	On superstable groups with residual properties}.
	\newblock Math. Log. Q. \textbf{53} (2007),
	no. 01, 19-26.

	\bibitem{houcine2}
	A. Ould Houcine, D. Vallino.
	\newblock {\em 	Algebraic and definable closure in free groups}.
	\newblock Ann. Inst. Fourier.  \textbf{66} (2016),
	no. 6, 2525-2563.


	\bibitem{HPQ}
	T. Hyttinen, G. Paolini and D. E. Quadrellaro.
	\newblock {\em A New Construction Principle}. 
	\newblock arXiv preprint arXiv:2505.10155.


	\bibitem{jarden}
	A. Jarden and S. Shelah.
	\newblock {\em Non-Forking Frames in Abstract Elementary Classes}.
	\newblock Ann. Pure Appl. Logic, 164:135-191, 2013.


	\bibitem{KM06}
	O. Kharlampovich and A. Myasnikov. 
	\newblock {\em Elementary theory of free non-abelian groups}.
	\newblock J. Algebra {\bf 302} (2006), no. 2, 451–552.

	\bibitem{armida1}
	M. Mazari-Armida.
	\newblock {\em Some stable non-elementary classes of modules}.
	\newblock J. Symb. Log. {\bf 88} (2023), no. 1, 93-117.
	
	\bibitem{armida2}
	M. Mazari-Armida.
	\newblock {\em Characterizing categoricity in several classes of modules}.
	\newblock J. Algebra {\bf 617} (2023), no. 1, 382-401.
	
	\bibitem{armida3}
	M. Mazari-Armida.
	\newblock {\em On superstability in the class of flat modules and perfect rings}.
	\newblock Proc. Am. Math. Soc. {\bf 149} (2021), no. 6, 2639--2654.
	
	\bibitem{armida4}
	M. Mazari-Armida.
	\newblock {\em Superstability, Noetherian rings and pure-semisimple rings}.
	\newblock  Ann. Pure Appl. Logic, {\bf 172} (2021), no. 3. 
	
	\bibitem{Armida_ros}
	M. Mazari-Armida and J. Rosick\'y.
	\newblock {\em Relative injective modules, superstability and noetherian categories}.
	\newblock J. Math. Log. \emph{to appear}.
	
	\bibitem{mekler}
	A. H. Mekler.
	\newblock {\em How to construct almost free groups}. 
	\newblock Can. J. Math. {\bf 32} (1980), no. 5, 1206-1228.
	
	\bibitem{mekler2}
	A. H. Mekler.
	\newblock {\em Almost-free groups in varieties}. 
	\newblock J. Algebra {\bf 145} (1992), no. 1, 128-142.
	
	\bibitem{MSh}
	A. H. Mekler and S. Shelah.
	\newblock {\em $L_{\infty, \omega}$-free algebras}. 
	\newblock Algebra Univers.  {\bf 26} (1989), 351-366.
	
	\bibitem{MSh2}
	A. H. Mekler and S. Shelah.
	\newblock {\em The consistency strength of “every stationary set reflects''}. 
	\newblock Isr. J. Math.  {\bf 67} (1989), no.~3, 353-366.
	
	\bibitem{MSh3}
	A. H. Mekler and S. Shelah.
	\newblock {\em Almost free algebras}. 
	\newblock Isr. J. Math.  {\bf 89} (1995), 237-259.

	
	\bibitem{Miasnikov}
	A. Miasnikov, E. Ventura and P. Weyl.
	\newblock {\em Algebraic extensions in free groups}. 
	\newblock Geometric Group Theory: Geneva and Barcelona Conferences, pp. 225-253, Birkhäuser, Basel.
				
				
	\bibitem{MPS22}
	B. M\"uhlherr, G. Paolini, and S. Shelah. 
	\newblock {\em First-order aspects of Coxeter groups}. 
	\newblock J. Algebra {\bf 595} (2022), 297–346.


	\bibitem{Perin}
	C. Perin. 
	\newblock {\em Elementary embeddings in torsion-free hyperbolic groups}. 
	\newblock Ann. Sci. Éc. Norm. Supé {\bf 44} (2011), no. 4, 631-681


	\bibitem{Poizat}
	B. Poizat.
	\newblock Groupes stables, avec types generiques reguliers
	\newblock J. Symb. Log. {\bf 48} (1983), no. 2, 339-355.
	
	\bibitem{zsela}
	Z. Sela.
	\newblock {\em Diophantine geometry over groups VI: The elementary theory of a free group}. 
	\newblock Geometric \& Functional Analysis {\bf 16} (2006), no. 3, 707-730.
	
	\bibitem{Sh}
	S. Shelah.
	\newblock {\em Classification theory: and the number of non-isomorphic models}.
	\newblock Studies in Logic and the Foundations of Mathematics, Vol. 92,  North-Holland Publishing Company, Amsterdam, 1978. 


	\bibitem{Sh2}
	S. Shelah.
	\newblock {\em Classification theory for abstract elementary classes. Volume 1}.
	\newblock Studies in Logic: Mathematical logic and foundations, College Publications, London, 2009.


	\bibitem{SaTr}
	J. Šaroch and J. Trlifaj.
	\newblock {\em Deconstructible abstract elementary classes of modules and categoricity}. 
	\newblock  Bull. Lond. Math. Soc. 56.12 (2024): 3854-3866.
		
	\bibitem{tent_book}
	K. Tent and M. Ziegler.
	\newblock {\em A course in model theory}.
	\newblock Cambridge University Press, Cambridge, 2012.
	
	\bibitem{Trlifaj}
	J. Trlifaj.
	\newblock {\em Abstract elementary classes induced by tilting and cotilting modules have finite character}.
	\newblock Proc. Am. Math. Soc. {\bf 137} (2009), no. 3, 1127-1133.

	
\end{thebibliography}
\end{document}